\documentclass[reqno,12pt]{amsart}

\usepackage{amsmath}
\usepackage{amsfonts}
\usepackage{amssymb}
\usepackage{eufrak}
\usepackage{amscd}
\usepackage{amsthm}
\usepackage{amstext}
\usepackage[all]{xy}

\newcommand{\id}{\operatorname{id}}

\newcommand{\ind}{\operatorname{ind}}

   \theoremstyle{plain}
   \newtheorem{thm}{Theorem}[section]
   
   \newtheorem{lem}[thm]{Lemma}
   \newtheorem{cor}[thm]{Corollary}
   \theoremstyle{definition}
   
   \newtheorem{defn}[thm]{Definition}
   \newtheorem{example}[thm]{Example}

   \newtheorem{remark}[thm]{Remark}

\title{Balanced pairs of operators and their relative index}

\author{V. Manuilov}

\date{}

\address{Moscow State University,
Leninskie Gory, Moscow, 
119991, Russia}

\email{manuilov@mech.math.msu.su}

\thanks{The author acknowledges partial support by the RFBR grant No. 16-01-00373.}

\begin{document}

\maketitle

\begin{abstract}
We show that the $K_1$ group of a $C^*$-algebra $A$ can be defined as homotopy classes of pairs, called {\it balanced}, of not necessarily unitary matrices over $A$ that have equal defects from being unitary. We also consider pairs of order zero pseudodifferential operators, not necessarily elliptic, with symbols being a balanced pair. A relative index is defined for such pairs of operators and it equals the topological index of the pair of their symbols.

\end{abstract}

\section{Introduction}

In \cite{M_K} we have discovered that the $K_0$ group for a $C^*$-algebra $A$ can be defined using more general elements than projections. Namely. one can consider homotopy classes of {\it pairs} $(a,b)$ of selfadjoint matrices with entries in $A$ that satisfy the relations $p(a)=p(b)$ for polynomials $p(t)=t(1-t)$ and $t^2(1-t)$ (or, equivalently, for all polynomials satisfying $p(0)=p(1)=0$). Genuine projections satisfy $p(a)=p(b)=0$, but this property is too restrictive. It suffices to require only that the defect from being a projection should be the same for $a$ and for $b$, but it doesn't need to vanish.  
In this paper we give a similar description for the $K_1$ group. It is generated by {\it balanced} pairs $(a,b)$, where $a,b$ need not to be unitaries, but their defect from being unitary should be the same.

Then we consider pairs $(D_1,D_2)$ of order zero pseudodifferential operators on a manifold $M$, such that their symbols $\sigma_1$, $\sigma_2$ are a balanced pair of matrix-valued functions on the cospherical bundle $S^*M$ over $M$. As $(\sigma_1,\sigma_2)$ represents an element in $K^1(S^*M)$, so its topological index is defined. We show that one can decompose $L^2(M)$ as an orthogonal direct sum $L^2(M)=H_1\oplus H_2$ in such a way that the restrictions of $D_1$ and of $D_2$ onto $H_2$ are almost the same, and the restrictions of $D_1$, $D_1^*$, $D_2$, $D_2^*$ onto $H_1$ are left-invertible modulo compact operators. The latter property allows to define a relative analytical index for the pair $((D_1)|_{H_1},(D_2)|_{H_1})$ and to show that it is equal to the topological index determined by the pair of symbols. We need hardly mention that neither the symbols $\sigma_1$, $\sigma_2$ have to be invertible, nor the operators $D_1$, $D_2$ have to be elliptic.

\section{$K_1$ group --- another description}

Let $A$ be a $C^*$-algebra. For two contractions $a,b\in A$, consider the following sets of relations 
\begin{equation}\label{rel1}
a^*a=b^*b;\quad aa^*=bb^*;\quad a(1-a^*a)=b(1-b^*b);\quad (1-aa^*)a=(1-bb^*)b
\end{equation}
and
\begin{equation}\label{rel2}
(a-b)c=(a^*-b^*)c=0 \mbox{ for\ } c \mbox{ being\ one\ of\ } 1-a^*a, 1-aa^*, 1-b^*b,1-bb^*. 
\end{equation}

Although we use the unit in these relations for convenience of notation, these relations make sense for non-unital $C^*$-algebras as well.  

\begin{lem}
Under the assumption that $a$ and $b$ are contractions, the sets of relations (\ref{rel1}) and (\ref{rel2}) are equivalent.

\end{lem}
\begin{proof}
(\ref{rel2}) easily (algebraically) follows from (\ref{rel1}). To prove the opposite, one needs to use the uniqueness of the positive square root in $C^*$-algebras. It follows from (\ref{rel2}) that 
$$
(1-a^*a)a^*a=a^*(1-aa^*)a=a^*(1-aa^*)b=(1-a^*a)a^*b=(1-a^*a)b^*b,
$$
therefore 
$$
(1-a^*a)(1-b^*b)=(1-a^*a)-(1-a^*a)b^*b=(1-a^*a)-(1-a^*a)a^*a=(1-a^*a)^2. 
$$
Passing to adjoints, we get 
$$
(1-b^*b)(1-a^*a)=(1-a^*a)^2.
$$ 
Interchanging $a$ and $b$, we get 
$$
(1-a^*a)(1-b^*b)=(1-b^*b)^2.
$$
Thus, $(1-a^*a)^2=(1-b^*b)^2$, hence $1-a^*a=1-b^*b$, and $a^*a=b^*b$. Similarly one can prove that $aa^*=bb^*$. The two other relations in (\ref{rel1}) can be shown algebraically.

\end{proof}

\begin{defn}
Pairs $(a,b)$ of contractions satisfying the relations (\ref{rel1}) or (\ref{rel2}) are called {\it balanced}.

\end{defn}

Two pairs, $(a_0,b_0)$ and $(a_1,b_1)$ of elements in $A$, are {\it homotopy equivalent} if there are paths $a=(a_t),b=(b_t):[0,1]\to A$, connecting $a_0$ with $a_1$ and $b_0$ with $b_1$ respectively, such that the pair $(a_t,b_t)$ is balanced for each $t\in[0,1]$.

Note that the pair $(0,0)$ is balanced. A pair $(a,b)$ is {\it homotopy trivial} if it is homotopy equivalent to $(0,0)$.

\begin{lem}\label{L1}
The pair $(a,a)$ is homotopy trivial for any $a\in A$.

\end{lem}
\begin{proof}
The linear homotopy $a_t=t\cdot a$ would do.

\end{proof}

\begin{lem}
If $\|a\|<1$, $\|b\|<1$ and the pair $(a,b)$ is balanced then it is homotopy trivial.

\end{lem}
\begin{proof}
The assumption implies that $a=b$.

\end{proof}

Let $M_n(A)$ denote the $n{\times n}$ matrix algebra over $A$. Two balanced pairs, $(a_0,b_0)$ and $(a_1,b_1)$, where $a_0,a_1,b_0,b_1\in M_n(A)$, are {\it equivalent} if there is a homotopy trivial pair $(a,b)$, $a,b\in M_m(A)$ for some integer $m$, such that the balanced pairs $(a_0\oplus a,b_0\oplus b)$ and $(a_1\oplus a,b_1\oplus b)$ are homotopy equivalent in $M_{n+m}(A)$. Using the standard inclusion $M_n(A)\subset M_{n+k}(A)$ (as the upper left corner) we may speak about equivalence of pairs of different matrix size.

Let $[(a,b)]$ denote the equivalence class of the pair $(a,b)$, $a,b\in M_n(A)$.

For two pairs, $(a,b)$, $a,b\in M_n(A)$, and $(c,d)$, $c,d\in M_m(A)$, set 
$$
[(a,b)]+[(c,d)]=[(a\oplus c,b\oplus d)]. 
$$
The result obviously doesn't depend on a choice of representatives. Also $[(a,b)]+[(c,d)]=[(a,b)]$ when $(c,d)$ is homotopy trivial.

\begin{lem}\label{commut}
The addition is commutative and associative.

\end{lem}
\begin{proof}
If $(u_t)_{t\in[0,1]}$ is a path of unitaries in $A$, $u_1=1$, $u_0=u$, then $[(u^*au,u^*bu)]=[(a,b)]$ for any $a,b\in A$, as the relations (\ref{rel1}) are not affected by unitary equivalence. The standard argument with a unitary path connecting $\left(\begin{smallmatrix}1&0\\0&1\end{smallmatrix}\right)$ and $\left(\begin{smallmatrix}0&1\\1&0\end{smallmatrix}\right)$ proves commutativity. A similar argument proves associativity.

\end{proof}

\begin{lem}
$[(a,b)]+[(b,a)]=[(0,0)]$ for any balanced pair $(a,b)$.

\end{lem}
\begin{proof}

By definition, $[(a,b)]+[(b,a)]=\left[\left(\left(\begin{smallmatrix}a&0\\0&b\end{smallmatrix}\right),\left(\begin{smallmatrix}b&0\\0&a\end{smallmatrix}\right)\right)\right]$.

Set 
\begin{equation}\label{U}
U_t=\left(\begin{smallmatrix}\cos t&-\sin t\\ \sin t&\cos t\end{smallmatrix}\right),
\end{equation}
$$
A=\left(\begin{smallmatrix}a&0\\0&b\end{smallmatrix}\right),\quad B_t=U_t^*\left(\begin{smallmatrix}a&0\\0&b\end{smallmatrix}\right)U_t. 
$$
We claim that if the pair $(a,b)$ is balanced then the pair $(A,B_t)$ is balanced for any $t$. If true, this implies that 
$$
\left[\left(\left(\begin{smallmatrix}a&0\\0&b\end{smallmatrix}\right),\left(\begin{smallmatrix}b&0\\0&a\end{smallmatrix}\right)\right)\right]=[(A,B_{\pi/2})]=[(A,B_0)]=\left[\left(\left(\begin{smallmatrix}a&0\\0&b\end{smallmatrix}\right),\left(\begin{smallmatrix}a&0\\0&b\end{smallmatrix}\right)\right)\right]=[0,0]
$$ 
by Lemma \ref{L1}. So, it remains to check that the relations (\ref{rel1}) hold for $(A,B_t)$. Note that if the pair $(a,b)$ is balanced then $\left(\begin{smallmatrix}a^*a&0\\0&b^*b\end{smallmatrix}\right)=a^*a\left(\begin{smallmatrix}1&0\\0&1\end{smallmatrix}\right)$. So,
$$
B_t^*B_t=U_t^*\left(\begin{smallmatrix}a^*&0\\0&b^*\end{smallmatrix}\right)U_t U_t^*\left(\begin{smallmatrix}a&0\\0&b\end{smallmatrix}\right)U_t=U_t^*a^*aU_t=A^*A;
$$
$$
B_t(1-B_t^*B_t)=U_t^*\left(\begin{smallmatrix}a&0\\0&b\end{smallmatrix}\right)(1-a^*a)U_t=U_t^*\left(\begin{smallmatrix}a(1-a^*a)&0\\0&a(1-a^*a)\end{smallmatrix}\right)U_t=A(1-A^*A).
$$

The two other relations in (\ref{rel1}) are proved in the same way.

\end{proof}

\begin{lem}
$[(a,b)]+[(a^*,b^*)]=[(0,0)]$ for any balanced pair $(a,b)$.

\end{lem}
\begin{proof}
Set 
$$
A_t=\left(\begin{smallmatrix}a&0\\0&1\end{smallmatrix}\right) U_t^*\left(\begin{smallmatrix}1&0\\0&a^*\end{smallmatrix}\right) U_t,\quad
B_t=\left(\begin{smallmatrix}b&0\\0&1\end{smallmatrix}\right) U_t^*\left(\begin{smallmatrix}1&0\\0&b^*\end{smallmatrix}\right) U_t,
$$
where $U_t$ is given by (\ref{U}). We claim that the pair $(A_t,B_t)$ is balanced for any $t$. Let us check the first relation in (\ref{rel1}) (other relations are checked similarly).
\begin{eqnarray*}
A_t^*A_t&=&C_t^*\left(\begin{smallmatrix}1&0\\0&a\end{smallmatrix}\right)C_t\left(\begin{smallmatrix}a^*a&0\\0&1\end{smallmatrix}\right)C_t^*\left(\begin{smallmatrix}1&0\\0&a^*\end{smallmatrix}\right)C_t\\
&=&C_t^*\left(\begin{smallmatrix}1&0\\0&a\end{smallmatrix}\right)\Bigl(\left(\begin{smallmatrix}1&0\\0&1\end{smallmatrix}\right)+(1-a^*a)\left(\begin{smallmatrix}-\cos^2 t&\cos t\sin t\\ \cos t\sin t&-\sin^2 t\end{smallmatrix}\right)\Bigr)\left(\begin{smallmatrix}1&0\\0&a^*\end{smallmatrix}\right)C_t\\
&=&C_t^*\left(\begin{smallmatrix}1&0\\0&a\end{smallmatrix}\right)\left(\begin{smallmatrix}1&0\\0&a^*\end{smallmatrix}\right)C_t+C_t^*\left(\begin{smallmatrix}1&0\\0&a\end{smallmatrix}\right)(1-b^*b)\left(\begin{smallmatrix}-\cos^2 t&\cos t\sin t\\ \cos t\sin t&-\sin^2 t\end{smallmatrix}\right)\left(\begin{smallmatrix}1&0\\0&a^*\end{smallmatrix}\right)C_t\\
&=&C_t^*\left(\begin{smallmatrix}1&0\\0&bb^*\end{smallmatrix}\right)C_t+C_t^*\left(\begin{smallmatrix}1&0\\0&b\end{smallmatrix}\right)(1-b^*b)\left(\begin{smallmatrix}-\cos^2 t&\cos t\sin t\\ \cos t\sin t&-\sin^2 t\end{smallmatrix}\right)\left(\begin{smallmatrix}1&0\\0&b^*\end{smallmatrix}\right)C_t= B_t^*B_t.
\end{eqnarray*}

One has $A_0=\left(\begin{smallmatrix}a&0\\0&a^*\end{smallmatrix}\right)$, $B_0=\left(\begin{smallmatrix}b&0\\0&b^*\end{smallmatrix}\right)$; $A_{\pi/2}=\left(\begin{smallmatrix}a^*a&0\\0&1\end{smallmatrix}\right)=B_{\pi/2}$. By Lemma \ref{L1}, we are done.

\end{proof}

We see that the equivalence classes of balanced pairs in matrix algebras over $A$ form an abelian group for any $C^*$-algebra $A$. Let us denote this group by $L_1(A)$.

Let now $A$ be unital.
Note that the pairs $(u,v)$, where $u,v\in A$ are unitaries, are patently balanced. The map 
$$
\iota([u])=[(u,1)]
$$
gives rise to a homomorphism $\iota:K_1(A)\to L_1(A)$.

Set 
\begin{equation}\label{unitary}
c=c(a,b)=1+b^*(a-b).
\end{equation}

\begin{lem}
Let $(a,b)$ be a balanced pair. Then
\begin{enumerate}
\item
$c(a,b)$ is unitary; similarly, $1+(a-b)b^*$ is unitary;
\item
$bc=a$;
\item
$b^*b$ commutes with $c$ and with $c^*$, hence with $f(c)$ for any continuous function $f$ on the spectrum of $c$; 
\item
$(1-b^*b)(c-1)=0=(c-1)(1-b^*b)$, hence $(1-b^*b)g(c)=0$ for any continuous function $g$ on the spectrum of $c$ with $g(1)=0$

\end{enumerate}

\end{lem}
\begin{proof}
(1) $c^*c=(1+(a-b)^*b)(1+b^*(a-b))=1+a^*b-b^*b+b^*a-b^*b+(a-b)^*bb^*(a-b)=1+a^*b+b^*a-2b^*b+(a-b)^*(bb^*-1)(a-b)+(a-b)^*(a-b)=
1+a^*b+b^*a-2b^*b+(a-b)^*(a-b)=1+a^*b+b^*a-2b^*b+a^*a+b^*b-a^*b-b^*a=1$, similarly one gets $cc^*=1$. The case of $1+(a-b)b^*$ can be checked in the same way.

(2)
$bc=b(1+b^*(a-b))=b+bb^*(a-b)=b-(1-bb^*)(a-b)+(a-b)=b+(a-b)=a$;

(3)
$b^*bc=b^*b+b^*bb^*a-(b^*b)^2$; $cb^*b=b^*b+b^*ab^*b-(b^*b)^2$, so it remains to show that $b^*bb^*a=b^*ab^*b$, which holds true: $b^*(bb^*)a=b^*(aa^*)a=b^*a(a^*a)=b^*a(b^*b)$. Also $b^*bc^*=(cb^*b)^*=(b^*bc)^*=c^*b^*b$.

(4)
$(1-b^*b)(c-1)=(1-b^*b)b^*(a-b)=b^*(1-bb^*)(a-b)=0$; $(c-1)(1-b^*b)=b^*(a-b)(1-b^*b)=0$.

\end{proof}

\begin{thm}\label{Thm}
The map $\iota:K_1(A)\to L_1(A)$ is a natural isomorphism for any unital $C^*$-algebra $A$.

\end{thm}
\begin{proof}
For a balanced pair $(a,b)$, set $\kappa(a,b)=c(a,b)$. Then it gives rise to a homomorphism $\kappa:L_1(A)\to K_1(A)$. We shall show that $\kappa$ is the inverse map for $\iota$.

Let $u$ be a unitary. Then $c(u,1)=1+(u-1)=u$, hence $\kappa\circ\iota([u])=[c(u,1)]=[u]$, so $\kappa\circ\iota=\id_{K_1(A)}$. 

Let $(a,b)$ be a balanced pair. Then $\iota\circ\kappa([(a,b)])=[(c(a,b),1)]$. We have to check that $[(c(a,b),1)]=[(a,b)]$. Equivalently, we check that the pair $\left(\left(\begin{smallmatrix}c&\\0&b\end{smallmatrix}\right),\left(\begin{smallmatrix}1&\\&a\end{smallmatrix}\right)\right)$ is homotopy trivial, where $c=c(a,b)$, $a=bc$.

Set $A=\left(\begin{smallmatrix}c&0\\0&b\end{smallmatrix}\right)$, $B_t=\left(\begin{smallmatrix}1&0\\0&b\end{smallmatrix}\right)U_t^*\left(\begin{smallmatrix}1&0\\0&c\end{smallmatrix}\right)U_t$, where $U_t$ is given by (\ref{U}).

As $B_0=\left(\begin{smallmatrix}1&0\\0&a\end{smallmatrix}\right)$, $B_{\pi/2}=A$, and as the pair $(A,A)$ is trivial, so it remains to check that the pair $(A,B_t)$ is balanced for any $t$. 

\begin{eqnarray*}
B_t^*B_t&=&U_t^*\left(\begin{matrix}1&0\\0&c^*\end{matrix}\right)U_t\left(\begin{matrix}1&0\\0&b^*b\end{matrix}\right)U_t^*
\left(\begin{matrix}1&0\\0&c\end{matrix}\right)U_t\\
&=&U_t^*\left(\begin{matrix}1&0\\0&c^*\end{matrix}\right)U_t\left(\left(\begin{matrix}1&0\\0&1\end{matrix}\right)-\left(\begin{matrix}0&0\\0&1-b^*b\end{matrix}\right)\right)U_t^*
\left(\begin{matrix}1&0\\0&c\end{matrix}\right)U_t\\
&=&\left(\begin{matrix}1&0\\0&1\end{matrix}\right)-U_t^*\left(\begin{matrix}1&0\\0&c^*\end{matrix}\right)U_t\left(\begin{matrix}0&0\\0&1-b^*b\end{matrix}\right)U_t^*
\left(\begin{matrix}1&0\\0&c\end{matrix}\right)U_t.
\end{eqnarray*}

Note that 
$$
C=
U_t^*\left(\begin{matrix}1&0\\0&c\end{matrix}\right)U_t=\left(\begin{matrix}1&0\\0&1\end{matrix}\right)+
\left(\begin{matrix}\sin^2t&\sin t\cos t\\\sin t\cos t&\cos^2 t\end{matrix}\right)(c-1),
$$
and, as $(1-b^*b)(c-1)=0$, so we have 
$$
C^*\left(\begin{matrix}0&0\\0&1-b^*b\end{matrix}\right)C=\left(\begin{matrix}0&0\\0&1-b^*b\end{matrix}\right),
$$
so 
$$
B_t^*B_t=\left(\begin{matrix}1&0\\0&1\end{matrix}\right)-\left(\begin{matrix}0&0\\0&1-b^*b\end{matrix}\right)=
\left(\begin{matrix}1&0\\0&b^*b\end{matrix}\right)=A^*A.
$$
The other relations are checked in the same way.

\end{proof}

%
%

\section{Nonunital case}

\begin{lem}
Let $A$ be a non-unital $C^*$-algebra, $A^+$ its unitalization. Then the inclusion $A\subset A^+$ induces an isomorphism $L_1(A)\to L_1(A^+)$.

\end{lem}
\begin{proof}
Let us prove surjectivity first. Since $L_1(A^+)\cong K_1(A^+)$, any element of $L_1(A^+)$ is, for some $n$, of the form $(u,1)$, where $u\in M_n(A^+)$ is unitary, $1=1_n\in M_n(A^+)$ is the unit, and without loss of generality we may assume that $u-1\in M_n(A)$. Set $A_n=M_n(A)$. Then $u,1\in A_n^+$ and $u-1\in A_n$. Fix $\delta\in(0,1/3)$ and let $f\in C(\mathbb S^1)$ be a continuous function with the following properties:
\begin{enumerate}
\item
$|f(z)|=1$ for any $z\in\mathbb S^1$;
\item
$f(z)=1$ for any $z\in\mathbb S^1$, for which $|z-1|<\delta$;
\item
$|f(z)-z|<\delta$ for any $z\in\mathbb S^1$.

\end{enumerate}

Then $f(u)\in A_n^+$ is unitary, $f(u)-1\in A_n$, and $\|f(u)-u\|<\delta$, hence $[f(u)]=[u]$ in $K_1(A^+)$ and $[(f(u),1)]=[(u,1)]$ in $L_1(A^+)$.

Let $g\in C(\mathbb S^1)$ satisfy the properties
\begin{enumerate}
\item
$g(z)=1$ for any $z\in\mathbb S^1$, for which $|z-1|>\delta$;
\item
$|g(z)|\leq 1$ for any $z\in\mathbb S^1$;
\item
$g(1)=0$,

\end{enumerate}
and let $g_t(z)=t+(1-t)g(z)$, $t\in[0,1]$. Then the pair $(f(u)g(u),g(u))$ is balanced for any $t\in[0,1]$ (all calculations are in $C(X)$: $|(fg)(z)|=|g(z)|=1$ when $|z-1|>\delta$, and $(fg)(z)=g(z)$ when $|z-1|<\delta$), hence $[(u,1)]=[(f(u)g(u),g(u))]$ in $L_1(A^+)$. But $g(u)\in A_n$, so $[(f(u)g(u),g(u))]\in L_1(A)$.

Now let us prove injectivity. Take $[(a,b)]\in L_1(A)$. Let $c$ be as in Lemma \ref{c}, and set $\tilde a=bf(c)$. Then $[(a,b)]=[(\tilde a,b)]$. 

Set $A=\left(\begin{matrix}\tilde a&0\\0&g(c)\end{matrix}\right)$, $B_t=\left(\begin{matrix}b&0\\0&g(c)\end{matrix}\right)U_t^*\left(\begin{matrix}f(c)&0\\0&1\end{matrix}\right)U_t$, where $U_t$ is given by (\ref{U}). Direct calculation allows to check that
\begin{enumerate}
\item
the pair $(A,B_t)$ is balanced for any $t$; 
\item
$B_0=A$; $B_{\pi/2}=\left(\begin{matrix}b&0\\0&f(c)g(c)\end{matrix}\right)$;
\item
$B_t$ is a matrix with coefficients in $A$ (not in $A^+$). 
\end{enumerate}
The first claim can be proved as in Theorem \ref{Thm}, and the other claims are trivial.

It follows that $[(A,B_{\pi/2})]=0$ in $L_1(A)$. As $A=\tilde a\oplus g(c)$, $B_{\pi/2}=b\oplus f(c)g(c)$, so $[(\tilde a,b)]+[(g(c),f(c)g(c))]=0$, i.e. $[(\tilde a,b)]=[(f(c)g(c),g(c))]$ in $L_1(A)$.

Now suppose that $[c]=0$ in $K_1(A)\cong L_1(A^+)$. Then, passing to matrices of greater size if necessary, we can connect $c$ and 1 by a path of unitaries $c_t$, $t\in[0,1]$. The pair $(f(c_t)g(c_t),g(c_t))$ is balanced for each $t\in[0,1]$, and, by definition, $(fg)(1)=g(1)=0$, hence $[(a,b)]=[(0,0)]$, which proves injectivity.   

\end{proof}

\begin{cor}
The groups $K_1(A)$ and $L_1(A)$ are naturally isomorphic for any $C^*$-algebra $A$.

\end{cor}

\section{Some examples}

Let us consider the case $A=\mathbb C$. It is easy to see that the following holds.

\begin{lem}
Let $(a,b)$ be a balanced pair of matrices, i.e. of automorphisms of a finitedimensional space $V$. There exists a decomposition $V=V_1\oplus V_2$, $V_2=V_1^\perp$, such that $1-a^*a$ and $1-aa^*$ have the form $\left(\begin{smallmatrix}\ast&0\\0&0\end{smallmatrix}\right)$, and $a-b$ has the form $\left(\begin{smallmatrix}0&0\\0&\ast\end{smallmatrix}\right)$ with respect to this decomposition. 

\end{lem}

But even in this case it is not so easy to describe all balanced pairs of numerical matrices. In general case, the structure of balanced pairs may be even more complicated. Here we give two examples for the case when $A$ is commutative.

\begin{example}\label{nontrivial}
Let $A=M_2(C(\mathbb S^1))$ be the algebra of $2{\times}2$-matrix-valued functions on a circle with the coordinate $t$, $t\in[0,\pi/2]$. Set $U(t)=\left(\begin{matrix}\cos t&-\sin t\\\sin t&\cos t\end{matrix}\right)$, $\alpha,\beta,\gamma:[0,\frac{\pi}{2}]\to\mathbb C$ with $\alpha(0)=\alpha(\frac{\pi}{2})=\beta(0)=\beta(\frac{\pi}{2})=\gamma(0)=\gamma(\frac{\pi}{2})=1$, $|\alpha(t)|=|\beta(t)|=1$, $|\gamma(t)|< 1$ for all $t\in(0,\frac{\pi}{2})$, 
$$
a(t)=U(t)^*\left(\begin{matrix}\alpha(t)&0\\0&\gamma(t)\end{matrix}\right)U(t), \qquad b(t)=U(t)^*\left(\begin{matrix}\beta(t)&0\\0&\gamma(t)\end{matrix}\right)U(t).
$$
The pair $(a,b)$ is balanced. Note that although $a(t)$ and $b(t)$ are diagonal at each $t$, $a$ and $b$ cannot be diagonalized as elements of $M_2(C(\mathbb S^1))$ (the eigenvectors of $a(t)$ and $b(t)$ cannot be continuous at $0$). 

\end{example}

\begin{example}
Let $X$ be a $2n$-dimensional manifold with boundary $\partial X=\mathbb S^{2n-1}$ and let $Y=X\cup \mathbb D^{2n}$ be $X$ glued with the $2n$-dimensional disc over the common boundary. Let $a,b:X\to \mathbb U_N$ be two maps into the unitary group of order $N$, such that $c=a|_{\partial X}=b|_{\partial X}:\mathbb S^{2n-1}\to\mathbb U_N$ represents a non-trivial element of $\pi_{2n-1}(\mathbb U_N)$. Then $c$ does not extend to a map $\mathbb D^{2n}\to\mathbb U_N$, but easily extends to a map $\bar c$ from $\mathbb D^{2n}$ to the set of $N$-dimensional matrices of norm $\leq 1$. Set 
$$
\bar a(x)=\left\lbrace \begin{array}{rl}a(x)&\mbox{if\ }x\in X;\\
\bar c(x)&\mbox{if\ }x\in\mathbb D^{2n}\end{array}\right.;
\qquad
\bar b(x)=\left\lbrace \begin{array}{rl}b(x)&\mbox{if\ }x\in X;\\
\bar c(x)&\mbox{if\ }x\in\mathbb D^{2n}.\end{array}\right.;
$$   
Then the pair $(\bar a,\bar b)$ is balanced in $M_N(C(Y))$.

\end{example}

\section{$\mathbb K$-balanced pairs of operators}

For an operator $A$ on a Hilbert space $H$, let $\dot A$ denote the class of $A$ in the Calkin algebra. 
Let $A$, $B$ be operators on a Hilbert space $H$. We call the pair $(A,B)$ balanced modulo compacts ($\mathbb K$-balanced) if the pair $(\dot A,\dot B)$ is balanced in the Calkin algebra. This means that the relations (\ref{rel1}) and (\ref{rel2}) hold modulo compacts. 

To study properties of $\mathbb K$-balanced pairs of operators we need the following corollaryof theKasparov's technical theorem \cite{KaspTT}. A set $X_1,\ldots,X_n$ of operators is symmetric if for every $i=1,\ldots,n$, $X_i^*$ is contained in this set.

\begin{lem}\label{Kasparov}
Let $\overline{X}=\{X_1,\ldots,X_n\}$ and $\overline{Y}=\{Y_1,\ldots,Y_m\}$ be two symmetric sets of contractions on a Hilbert space $H$. Suppose that $X_iY_j$ is compact for any $i$ and $j$. Then, for any $\varepsilon>0$, there exists a projection $P$ in $H$ such that
\begin{enumerate}
\item
$\|PX_i-X_i\|<\varepsilon$ for $i=1,\ldots,n$;
\item
$\|\dot P\dot Y_j\|<\varepsilon$ for $j=1,\ldots,m$.
\end{enumerate}

\end{lem}

\begin{proof}
Let $\mathcal A$ and $\mathcal B$ be the $C^*$-subalgebras in the Calkin algebra $\mathbb Q(H)$ generated by the sets $\dot X_1,\ldots,\dot X_n$ and $\dot Y_1,\ldots,\dot Y_m$ respectively. By the Kasparov's technical theorem, there exists $m\in\mathbb Q(H)$ such that $ma=a$, $mb=0$ for any $a\in\mathcal A$ and any $b\in\mathcal B$. Let $M\in\mathbb B(H)$ be a lift for $m$, i.e. an operator on $H$ such that $\dot M=m$. Without loss of generality we may assume that $M$ satisfies $0\leq M\leq 1$.

Let $Q=E_{(1-\varepsilon,1]}(M)$ be the spectral projection of $M$ corresponding to the set $(1-\varepsilon,1]$.  
As $Q\leq M+\varepsilon 1$, it follows from $\dot M\dot Y=0$ that $\|\dot Q\dot Y\|<\varepsilon$. The latter estimate will not change if we replace $Q$ by its compact perturbation $P$.

Let us write operators on $H$ as matrices with respect to the decomposition $H=QH\oplus(1-Q)H$. Then $M=\left(\begin{matrix}M_1&0\\0&M_2\end{matrix}\right)$ with $(1-\varepsilon)1\leq M_1\leq 1$ and $0\leq M_2\leq (1-\varepsilon)1$.

By Kasparov's technical theorem, $MX-X$ is compact for any $X$ from $\overline{X}$. Write $X$ as $X=\left(\begin{matrix}x_{11}&x_{12}\\x_{21}&x_{22}\end{matrix}\right)$. Then
$$
X-MX=\left(\begin{matrix}1-M_1&0\\0&1-M_2\end{matrix}\right)\left(\begin{matrix}x_{11}&x_{12}\\x_{21}&x_{22}\end{matrix}\right)
$$
is compact. Since $1-M_2$ is invertible, $x_{21}$ and $x_{22}$ are compact. It follows from symmetricity of $\overline{X}$ that $x_{12}$ is compact as well. So, 
$$
QX-X=-\left(\begin{matrix}0&x_{12}\\x_{21}&x_{22}\end{matrix}\right),
$$
therefore, for any $\varepsilon>0$ there is a projection $Q_0$ onto a finitedimensional subspace $H_0\subset (1-Q)H$ such that $\|Q_0x_{21}-x_{21}\|<\varepsilon/3$ and $\|Q_0x_{22}-x_{22}\|<\varepsilon/3$ for all $X$ from $\overline{X}$. Set 
$P=Q+Q_0$. Then $P$ is a compact perturbation of $Q$, and $\|PX-X\|<\varepsilon$ for each $X\in\overline{X}$.  

\end{proof}

\begin{thm}\label{H}
Let $(A,B)$ be a $\mathbb K$-balanced pair of operators on a Hilbert space $H$. For any $\varepsilon>0$ there exists a decomposition $H=H_1\oplus H_2$ with the following properties, where we write operators on $H$ as matrices with respect to this decomposition.
\begin{enumerate}
\item 
$\|A_{ij}-B_{ij}\|<\varepsilon$ for any $(i,j)\neq (1,1)$;
\item
$C_{ij}$ is of the form $D+K$ with $\|D\|<\varepsilon$ and $K$ compact for any $(i,j)\neq(2,2)$, where $C$ is one of the four operators: $1-A^*A$, $1-AA^*$, $1-B^*B$, $1-BB^*$.
\end{enumerate}

\end{thm}
\begin{proof}
Apply Lemma \ref{Kasparov} to the sets $\overline{X}=\{A-B,A^*-B^*\}$ and $\overline{Y}=\{1-A^*A,1-AA^*,1-B^*B,1-BB^*\}$, and set $H_1=PH$, $H_2=(1-P)H$. 

\end{proof}

\begin{lem}
Let $(A,B)$ be a $\mathbb K$-balanced pair of operators on a Hilbert space $H=H_1\oplus H_2$, where $H_1$ and $H_2$ satisfy the conclusion of Theorem \ref{H} for some $\varepsilon$. Then 
\begin{enumerate}
\item
$\|\dot A_{11}^*\dot A_{11}-\dot B_{11}^*\dot B_{11}\|<2\varepsilon$, $\|\dot A_{11}\dot A_{11}^*-\dot B_{11}\dot B_{11}^*\|<2\varepsilon$;
\item
$\|(\dot B_{11}-\dot A_{11})(1-\dot A_{11}^*\dot A_{11})\|<4\varepsilon$, $\|(\dot B_{11}-\dot A_{11})^*(1-\dot A_{11}\dot A_{11}^*)\|<4\varepsilon$.
\end{enumerate} 

\end{lem}
\begin{proof}
Both claims follow from $\dot A^*\dot A=\dot B^*\dot B$.

(1) As $(\dot A^*\dot A)_{11}=\dot A_{11}^*\dot A_{11}+\dot A_{21}^*\dot A_{21}$, $(\dot B^*\dot B)_{11}=\dot B_{11}^*\dot B_{11}+\dot B_{21}^*\dot B_{21}$, so $\|\dot A_{11}^*\dot A_{11}-\dot B_{11}^*\dot B_{11}\|=\|\dot A_{21}^*\dot A_{21}-\dot B_{21}^*\dot B_{21}\|<2\varepsilon$. The second estimate in (1) is proved similarly.

(2) As $(\dot A^*\dot A)_{12}=\dot A_{11}^*\dot A_{12}+\dot A_{12}^*\dot A_{22}$, $(\dot B^*\dot B)_{12}=\dot B_{11}^*\dot B_{12}+\dot B_{12}^*\dot B_{22}$, so $\|A_{11}^*\dot A_{12}-\dot B_{11}^*\dot B_{12}\|<2\varepsilon$. This implies that $\|(A_{11}^*-\dot B_{11}^*)\dot A_{12}\|<3\varepsilon$, hence $\|(A_{11}^*-\dot B_{11}^*)\dot A_{12}\dot A_{12}^*\|<3\varepsilon$.

As $(\dot B-\dot A)^*(1-\dot A\dot A^*)=0$, so $\|(\dot B_{11}-\dot A_{11})^*(1-\dot A\dot A^*)_{11}\|<\varepsilon$. Then
$\|(\dot B_{11}-\dot A_{11})^*(1-\dot A_{11}\dot A_{11}^*)\|<\|(\dot B_{11}-\dot A_{11})^*\dot A_{12}\dot A_{12}^*\|+\varepsilon<4\varepsilon$. 
The second estimate in (2) is proved similarly.

\end{proof}

\section{Relative index}

Let $(A,B)$ be a $\mathbb K$-balanced pair, and let the decomposition $H=H_1\oplus H_2$ satisfies the conclusion of Theorem \ref{H} for some $\varepsilon\in(0,1/20)$, i.e. 
\begin{enumerate}
\item 
$\|A_{ij}-B_{ij}\|<\varepsilon$ for any $(i,j)\neq (1,1)$;
\item
$C_{ij}$ is of the form $D+K$ with $\|D\|<\varepsilon$ and $K$ compact for any $(i,j)\neq(2,2)$, where $C$ is one of the four operators: $1-A^*A$, $1-AA^*$, $1-B^*B$, $1-BB^*$.
\end{enumerate}
In particular, this means that
\begin{enumerate}
\item
$\|A|_{H_2}-B|_{H_2}\|<\varepsilon$;
\item
$A|_{H_1}$ and $B|_{H_1}$ are isometries up to $\varepsilon$ modulo compacts, i.e. $\|\dot X^*\dot X-\dot 1_{H_1}\|<\varepsilon$, where $X$ is either $A|_{H_1}$ or $B|_{H_1}$.
\end{enumerate}

Then we can define a relative index of the pair $(A,B)$ as follows. Let us write operators as matrices with respect to the direct sum $H_1\oplus H_2$. Then $X|_{H_1}=\left(\begin{smallmatrix}X_1\\X_2\end{smallmatrix}\right)$. For convenience we write here $X_1$ instead of $X_{11}$ and $X_2$ instead of $X_{21}$.

Note that $A|_{H_1}$ and $B|_{H_1}$ behave like Fredholm operators, but their ranges may be completely different. To compare them, take one more operator $C=\left(\begin{smallmatrix}C_1\\C_2\end{smallmatrix}\right):H_1\to H$ with the following properties (where $X$ is either $A$ or $B$):
\begin{enumerate}
\item[(C1)]
$\|C_2-X_2\|<\varepsilon$;
\item[(C2)]
$\|\dot C_1^*\dot C_1-\dot X_1^*\dot X_1\|<2\varepsilon$; $\|\dot C_1\dot C_1^*-\dot X_1\dot X_1^*\|<2\varepsilon$;
\item[(C3)]
$\|(\dot C_1-\dot X_1)(1-\dot X_1^*\dot X_1)\|<4\varepsilon$ and $\|(\dot C_1-\dot X_1)^*(1-\dot X_1\dot X_1^*)\|<4\varepsilon$.
\end{enumerate}

Note that such operators $C$ exist. For example, one may take $C=A|_{H_1}$ or $C=B|_{H_1}$.

As $C^*=(C_1^*,C_2^*):H\to H_1$ has range $H_1$, so the compositions $C^*\circ A|_{H_1}$, $C^*\circ B|_{H_1}$ are operators on $H_1$. 

\begin{lem}\label{Lemma1F}
Operators $C^*\circ A|_{H_1}$ and $C^*\circ B|_{H_1}$ are Fredholm.

\end{lem}
\begin{proof}
We have $F=C^*\circ A|_{H_1}=C_1^*A_1+C_2^*A_2$. Set $G=C_1^*A_1+A_2^*A_2$, then $\|F-G\|<\varepsilon$. We have $\dot G=\dot C_1^*\dot A_1+1-\dot A_1^*\dot A_1=1+(\dot C_1^*-\dot A_1^*)\dot A_1$. Then, using $\|\dot C_1(1-\dot A_1^*\dot A_1)-\dot A_1(1-\dot A_1^*\dot A_1)\|<4\varepsilon$ and $\|\dot C_1\dot C_1^*-\dot A_1\dot A_1^*\|<2\varepsilon$, we get 
\begin{eqnarray*}
\|\dot G^*\dot G-1\|&=&\|\dot A_1^*\dot C_1+\dot C_1^*\dot A_1-2\dot A_1^*\dot A_1+\dot A_1^*\dot C_1\dot C_1^*\dot A_1-\dot A_1^*\dot A_1\dot C_1^*\dot A_1-\dot A_1^*\dot C_1\dot A_1^*\dot A_1+\dot A_1^*\dot A_1\dot A_1^*\dot A_1\|\\
&<&\|\dot A_1^*\dot C_1(1-\dot A_1^*\dot A_1)+(1-\dot A_1^*\dot A_1)\dot C_1^*\dot A_1-2\dot A_1^*\dot A_1+2\dot A_1^*\dot A_1\dot A_1^*\dot A_1\|+2\varepsilon\\
&<&\|\dot A_1^*\dot A_1(1-\dot A_1^*\dot A_1)+(1-\dot A_1^*\dot A_1)\dot A_1^*\dot A_1-2\dot A_1^*\dot A_1+2\dot A_1^*\dot A_1\dot A_1^*\dot A_1\|+10\varepsilon=10\varepsilon.
\end{eqnarray*}
Hence, $\|\dot F^*\dot F-1\|<11\varepsilon$. Similarly, $\|\dot F\dot F^*-1\|<11\varepsilon$, so $F$ is Fredholm. The same proof works for the second operator. 

\end{proof}    

Define $\ind(A,B)$ by $\ind(A,B)=\ind (C^*\circ A|_{H_1})-\ind (C^*\circ B|_{H_1})$.

\begin{lem}\label{Lemma2F}
$\ind(A,B)$ does not depend on $C$ when $C$ satisfies (C1)-(C3).

\end{lem} 
\begin{proof}
Without loss of generality we may take $C_2=A_2$. Then $(C^*\circ A|_{H_1})^\cdot=1+(\dot C_1^*-\dot A_1^*)\dot A_1$ and 
$\|(C^*\circ B|_{H_1})^\cdot-(1+\dot C_1^*\dot B_1-\dot A_1^*\dot A_1)\|<\varepsilon$. Let us check that 
$$
\ind (C^*\circ A|_{H_1})-\ind (C^*\circ B|_{H_1})=\ind ((A|_{H_1})^*\circ A|_{H_1})-\ind ((A|_{H_1})^*\circ B|_{H_1}).
$$
By multiplicativity of index, this is equivalent to 
$$
\ind(C^*\circ A|_{H_1})((A|_{H_1})^*\circ B|_{H_1})=\ind(C^*\circ B|_{H_1})((A|_{H_1})^*\circ A|_{H_1}),
$$
or, as $\|\dot A_1^*\dot A_1+\dot A_2^*\dot A_2-1\|<\varepsilon$, to 
\begin{equation}\label{2indices}
\ind(1+(C_1^*-A_1^*)A_1)(1+A_1^*(B_1-A_1))=\ind(1+C_1^*B_1-A_1^*A_1).
\end{equation}
Direct calculation shows that
$$
(1+(\dot C_1^*-\dot A_1^*)\dot A_1)(1+\dot A_1^*(\dot B_1-\dot A_1))=1+\dot C_1^*\dot A_1-\dot A_1^*\dot A_1+\dot A_1^*\dot B_1-\dot A_1^*\dot A_1+(\dot C_1^*-\dot A_1^*)\dot A_1\dot A_1^*(\dot B_1-\dot A_1); 
$$
\begin{eqnarray*}
&&\|(1+(\dot C_1^*-\dot A_1^*)\dot A_1)(1+\dot A_1^*(\dot B_1-\dot A_1))-(1+\dot C_1^*\dot B_1-\dot A_1^*\dot A_1)\|\\
&=&\|\dot C_1^*\dot A_1-2\dot A_1^*\dot A_1+\dot A_1^*\dot B_1+(\dot C_1^*-\dot A_1^*)\dot A_1\dot A_1^*(\dot B_1-\dot A_1)-\dot C_1^*\dot B_1+\dot A_1^*\dot A_1\|\\
&\leq&\|\dot C_1^*\dot A_1-2\dot A_1^*\dot A_1+\dot A_1^*\dot B_1+(\dot C_1^*-\dot A_1^*)(\dot B_1-\dot A_1)-\dot C_1^*\dot B_1+\dot A_1^*\dot A_1\|+4\varepsilon=4\varepsilon.
\end{eqnarray*}
Therefore, the operators in the both sides of (\ref{2indices}) coinside up to $4\varepsilon$ modulo compacts, thus they have the same index.

\end{proof}

\begin{cor}
$\ind(A,B)=\ind(1+B_1^*(A_1-B_1))$.

\end{cor}
\begin{proof}
Take $C_1=A_1$. Then 
\begin{eqnarray*}
\ind(A,B)&=&\ind((A|_{H_1})^*\circ A|_{H_1})-\ind((A|_{H_1})^*\circ B|_{H1})\\
&=&-\ind(1+A_1^*(B_1-A_1))=\ind(1+(B_1-A_1)^*A_1)\\
&=&\ind(1+B_1^*A_1-A_1^*A_1)=\ind(1+B_1^*A_1-B_1*B_1)\\
&=&\ind(1+B_1^*(A_1-B_1)).
\end{eqnarray*}

\end{proof}

Now let us show that $\ind(A,B)$ does not depend also on the decomposition $H_1\oplus H_2$. Let $F=c(A,B)=1+B^*(A-B)$ be the operator on $H\oplus H$ defined by the formula (\ref{unitary}).  Independence on $H_1$ is implied by the following theorem.

\begin{thm}
$\ind(A,B)=\ind(1+B^*(A-B))$.

\end{thm} 
\begin{proof}
The proof follows from the estimate
$$
\left\|B^*(A-B)-\left(\begin{matrix}B_1^*(A_1-B_1)&0\\0&0\end{matrix}\right)\right\|<6\varepsilon. 
$$

\end{proof}

\begin{remark}
Note that if, by some reason, $A_{21}$ is compact (or small plus compact) then $A|_{H_1}$ and $B|_{H_1}$ are Fredholm even without compositions with $C$, and there is no need in Lemmas \ref{Lemma1F} and \ref{Lemma2F}.

\end{remark}

\section{Application to pseudodifferential operators}

Let $D_1$, $D_2$ be two order zero pseudodifferential operators on a manifold $M$, with symbols $\sigma(D_1)$, $\sigma(D_2)$. Let $\sigma_1$, $\sigma_2$ be the restrictions of $\sigma(D_1)$ and $\sigma(D_2)$ onto the cospherical bundle $S^*M$. If $(\sigma_1,\sigma_2)$ is a balanced pair in matrices over $C(S^*M)$ then $(D_1,D_2)$ (more exactly, their compact perturbations) is a $\mathbb K$-balanced pair of operators. As the pair $(\sigma_1,\sigma_2)$ is balanced, it determines a class in $K^1(S^*M)$. The standard construction allows to define an integer-valued topological index for such pairs.

\begin{thm}
If $(\sigma_1,\sigma_2)$ is balanced in matrices over $C(S^*M)$ then $\ind(D_1,D_2)=\ind(\sigma_1,\sigma_2)$.

\end{thm} 
\begin{proof}
If $(\sigma_1,\sigma_2)$ is balanced then compact perturbations of $D_1$ and $D_2$ give a $\mathbb K$-balanced pair of operators ($D_1$ and $D_2$ satisfy the relations (\ref{rel1}) or (\ref{rel2}) modulo compacts, but may not be contractions; the latter can be remedied by compact perturbations). Then the operator $U(D_1,D_2)$ (\ref{unitary}) is pseudodifferential with the symbol $U(\sigma_1,\sigma_2)$, hence $\ind U(D_1,D_2)=\ind U(\sigma_1,\sigma_2)$. 

\end{proof}

Let us consider some simple examples. 

\begin{example}
Let $M_1$, $M_2$ be two manifolds with the same boundary $\partial M_1=\partial M_2=N$, and let $M=M_1\cup_N M_2$. Let $(\sigma_1,\sigma_2)$ be a balanced pair in $M_n(C(S^*M))$ such that $\sigma_1=\sigma_2$ on $M_2$, and $\sigma_1,\sigma_2$ unitary on $M_1$. Then one can take $H_1=L^2(M_1)$, $H_2=L^2(M_2)$ (if necessary, one can take a finite codimension subspace in $L^2(M_2)$ as $H_2$). The relative index $\ind(D_1,D_2)$ in this case equals the relative index of the restrictions of $D_1$ and $D_2$ on $M_1$.

\end{example}

\begin{example}
Let $p\in M_n(S^*M)$ be a projection, $P$ a pseudodifferential operator with the symbol $p$. Suppose that $p$ commutes with $\sigma_1$ and with $\sigma_2$, and that $(1-p)(\sigma_1-\sigma_2)=0$. Set $\sigma'_1=p\sigma_1 p$, $\sigma'_2=p\sigma_2 p$, and let $D'_1$, $D'_2$ be pseudoddifferential operators with symbols $\sigma'_1$ and $\sigma'_2$ respectively. Then we are in the setting of operators in subspaces \cite{SSS}.

\end{example}

\begin{example}
Let $x_0\in M$, and let $\sigma_1|_{x=x_0}=\sigma_1|_{x=x_0}=0$. Then we are in the setting of elliptic operators on a noncompact manifold, or on a manifold with a singularity at $x_0$.

\end{example}

The examples above reduce to known cases, but the next example seems to be new.

\begin{example}
Let $a(t)$, $b(t)$, $t\in[0,\pi/2]$, be as in Example \ref{nontrivial}. Let $M=\mathbb S^1$ (with the points $0$ and $\pi/2$ glued together), then $S^*M=\{(t,i):t\in\mathbb S^1, i=\pm 1\}$. Set 
$$
\sigma_1(t,1)=a(t),\quad \sigma_2(t,1)=b(t);\qquad \sigma_1(t,-1)=\sigma_2(t,-1)=\left(\begin{smallmatrix}1&0\\0&1\end{smallmatrix}\right).
$$ 
Then the relative index for $D_1$, $D_2$ having symbols $\sigma_1$, $\sigma_2$ respectively, is well defined and can be evaluated from the functions $\alpha$ and $\beta$.

\end{example}

\end{document}